\documentclass[10pt]{amsart}
\usepackage{amsmath, amssymb}
 \usepackage{mathrsfs}
\newcommand{\no}[1]{#1}
\renewcommand{\no}[1]{}
\no{\usepackage{times}\usepackage[subscriptcorrection, slantedGreek, nofontinfo]{mtpro}
\renewcommand{\Delta}{\upDelta}}
\usepackage{color}

\date{\today}
\newtheorem{theorem}{Theorem}[section]

\newtheorem{lemma}{Lemma}[section]

\newtheorem{corollary}{Corollary}[section]

\theoremstyle{remark}
\newtheorem{remark}{Remark}[section]

\numberwithin{equation}{section}

\title[Logarithmic stability in determining two coefficients]{Logarithmic stability in determining two coefficients in a dissipative wave equation. Extensions to clamped Euler-Bernoulli beam and heat equations}

\author[Ka\"{\i}s Ammari]{Ka\"{\i}s Ammari}
\address{UR Analysis and Control of Pde, UR 13ES64, Department of Mathematics, Faculty of Sciences of Monastir, University of Monastir, 5019 Monastir, Tunisia }
\email{kais.ammari@fsm.rnu.tn}

\author[Mourad Choulli]{Mourad Choulli}
\address{Institut \'Elie Cartan de Lorraine, UMR CNRS 7502, Universit\'e de Lorraine, Boulevard des Aiguillettes, BP 70239, 54506 Vandoeuvre les Nancy cedex - Ile du Saulcy, 57045 Metz cedex 01, France}
\email{mourad.choulli@univ-lorraine.fr}

\date{}

\begin{document}

\begin{abstract}
We are concerned with the inverse problem of determining both the potential and the damping coefficient  in a dissipative wave equation from boundary measurements. We establish stability estimates of logarithmic type when the measurements are given by the operator who maps the initial condition to Neumann boundary trace of the solution of the corresponding initial-boundary value problem. We build a method combining an observability inequality together with a spectral decomposition. We also apply this method to a clamped Euler-Bernoulli beam equation. Finally, we indicate how the present approach can be adapted to a heat equation.


\end{abstract}

\maketitle

\tableofcontents


\section{Introduction}
We consider the following initial-boundary value problem (abbreviated to IBVP in the sequel) for the wave equation:
\begin{equation}\label{wd}
\left\{
\begin{array}{lll}
 \partial _t^2 u - \Delta u + q(x)u + a(x) \partial_t u = 0 \;\; &\mbox{in}\;   Q=\Omega \times (0,\tau), 
 \\
u = 0 &\mbox{on}\;  \Sigma =\partial \Omega \times (0,\tau), 
\\
u(\cdot ,0) = u_0,\; \partial_t u (\cdot ,0) = u_1,
\end{array}
\right.
\end{equation}
where $\Omega \subset \mathbb{R}^n$, $n\geq 1$, is a bounded domain with $C^2$-smooth boundary $\partial \Omega$ and $\tau > 0$. 

\smallskip
We assume in this text that the coefficients $q$ and $a$ are real-valued.

\smallskip
Under the assumption that $q,a\in L^\infty (\Omega  )$, for each $\tau >0$ and $\left(\begin{array}{cc}u_0\\u_1\end{array}\right)\in H_0^1(\Omega )\times L^2(\Omega )$, the IBVP \eqref{wd} has a unique solution $u_{q,a}\in C([0,\tau ],H_0^1(\Omega ))$ such that $\partial _tu_{q,a}\in C([0,\tau ],L^2(\Omega ))$ (e.g.  \cite[pages 699-702]{DL}). On the other hand, by a  classical energy estimate, we have 
\[
\|u_{q,a}\|_{C([0,\tau ],H_0^1(\Omega ))}+\|\partial _t u_{q,a}\|_{C([0,\tau ],L^2(\Omega ))}\leq C(\|u_0\|_{1,2}+\|u_1\|_0).
\]
Here and henceforth, $\| \cdot \|_p$ and $\|\cdot \|_{s,p}$, $1\leq p\leq \infty$, $s\in \mathbb{R}$, denote respectively the usual $L^p$-norm and the $W^{s,p}$-norm. 

\smallskip
We note that the constant $C$ above is a non decreasing function of $\|q\|_\infty +\|a\|_\infty$.

\medskip
Now, since $u_{q,a}$ coincides with the solution of the IBVP \eqref{wd} in which $- q(x)u_{q,a} -a(x) \partial_t u_{q,a}$ is seen as a right-hand side, we can apply \cite[Theorem 2.1]{LLT} to get that $\partial _\nu u_{q,a}$, the derivative of the $u_{q,a}$ in the direction of $\nu$, the unit outward normal vector to $\partial \Omega$, belongs to $L^2(\Sigma )$. Additionaly, the mapping
\[
\left(\begin{array}{cc}u_0\\ u_1\end{array}\right)\in H_0^1(\Omega )\times L^2(\Omega ) \longrightarrow \partial _\nu u_{q,a}\in L^2(\Sigma )
\]
defines a bounded operator.

\smallskip
Let $\Gamma$ be a non empty open subset of $\partial \Omega$ and $\Upsilon =\Gamma \times (0,\tau )$. To $q,a\in L^\infty (\Omega )$, we associate the initial-to-boundary (abbreviated to IB in the following) operator $\Lambda _{q,a}$ defined by
\[
\Lambda _{q,a}: \left(\begin{array}{cc}u_0\\ u_1\end{array}\right)\in H_0^1(\Omega )\times L^2(\Omega ) \longrightarrow \partial _\nu u_{q,a}{_{|\Upsilon}}\in L^2(\Upsilon ).
\]
Clearly, from the preceding discussion, $\Lambda _{q,a}\in \mathscr{B}\left(H_0^1(\Omega )\times L^2(\Omega ),L^2(\Upsilon )\right)$.

\smallskip
We also consider two partial IB operators $\Lambda _q$ and $\tilde{\Lambda} _{q,a}$ which are given by
\[
\Lambda _q(u_0)=\Lambda_{q,0}\left(\begin{array}{cc}u_0\\ 0\end{array}\right),\quad \tilde{\Lambda} _{q,a}(u_1)=\Lambda_{q,a}\left(\begin{array}{cc}0\\ u_1\end{array}\right).
\]
Therefore, $\Lambda _q\in \mathscr{B}\left(H_0^1(\Omega ),L^2(\Upsilon )\right)$ and $\tilde{\Lambda} _{q,a} \in \mathscr{B}\left( L^2(\Omega ),L^2(\Upsilon )\right)$.

\smallskip
Next, we see that $\partial _tu$ is the solution of the IBVP \eqref{wd} corresponding to the initial conditions $u_1$ and $\Delta u_0-qu_0-au_1$. Hence, repeating the preceding analysis with $\partial _tu$ in place of $u$, we get   
\[
\Lambda _{q,a}\in \mathscr{B}\left(\left[H_0^1(\Omega )\cap H^2(\Omega )\right]\times H_0^1(\Omega ), H^1((0,\tau ),L^2(\Gamma ))\right).
\] 
Consequently, 
\begin{align*}
&\Lambda _q\in \mathscr{B}\left(H_0^1(\Omega )\cap H^2(\Omega ), H^1((0,\tau ),L^2(\Gamma ))\right),
\\
&\tilde{\Lambda} _{q,a} \in \mathscr{B}\left( H_0^1(\Omega ),H^1((0,\tau ),L^2(\Gamma ))\right).
\end{align*}

We are interested in the stability issue for the inverse problem consisting in the determination of both the potential $q$ and the damping coefficient $a$, appearing in the IBVP \eqref{wd}, from the IB map $\Lambda_{q,a}$. We succeed in proving logarithmic stability estimates  of determining $q$ from $\Lambda _q$, $a$ from  $\widetilde{\Lambda}_{q,a}$ and $(q,a)$ from $\Lambda_{q,a}$.

\medskip
We introduce the unbounded operators, defined on $H_0^1(\Omega )\times L^2(\Omega )$, as follows
\[ 
\mathcal{A}_0=\left( 
\begin{array}{cc}
0 & I  \\
\Delta  & 0  \\
 \end{array} 
 \right),\;\; D(\mathcal{A}_0)=\left[H^2(\Omega )\cap H_0^1(\Omega )\right]\times H_0^1(\Omega )
 \]
 and $\mathcal{A}=\mathcal{A}_{q,a}=\mathcal{A}_0+\mathcal{B}$ with $D(\mathcal{A})=D(\mathcal{A}_0)$, where
 \[ 
\mathcal{B}=\mathcal{B}_{q,a}=\left( 
\begin{array}{cc}
0 & 0 \\
-q & -a  \\
 \end{array} 
 \right).
 \]
Let 
\[
\mathcal{C}: D(\mathcal{A}_0)\rightarrow L^2(\Sigma ): \left(\begin{array}{cc}\varphi \\ \psi\end{array}\right)\longrightarrow \partial _\nu\varphi .
\]

Since we deal with the wave equation, it is necessary to make assumptions on $\Gamma$ and $\tau$ in order to guarantee that our system is observable. To this end, we assume that $\Gamma$ is chosen in such a way that there is $\tau _0$ such that the pair $(\mathcal{A},\mathcal{C})$ is exactly  observable for any $\tau \geq \tau _0$. We formulate the precise definition of exact observability in the next section in an abstract framework.

\smallskip
We give sufficient conditions ensuring that the pair $(\mathcal{A},\mathcal{C})$ is exactly observable. We fix $x_0\in \mathbb{R}^n\setminus \overline{\Omega}$ and we set
\[
\Gamma _0 =\{x\in \partial \Omega ;\; \nu (x)\cdot (x-x_0)> 0\}\;\; \mbox{and}\;\; d=\max_{x\in \overline{\Omega}}|x-x_0|.
\]
Let us assume that $\Gamma \supset \Gamma _0$. Following \cite[Theorem 7.2.3, page 233]{tucsnakweiss}, $(\mathcal{A}_0,\mathcal{C})$ is exactly observable with $\tau \geq \tau _0=2d$. In light  of \cite[Theorem 7.3.2, page 235]{tucsnakweiss} and the remark following it, we conclude that $(\mathcal{A},\mathcal{C})$ is also exactly observable for $\tau \geq \tau _0$, again with $\tau _0=2d$.

\smallskip
We mention that sharp sufficient conditions on $\Gamma$ and $\tau _0$ were given in a work by Bardos, Lebeau and Rauch \cite{blr}.

\smallskip
Unless otherwise stated, for sake of simplicity, all operator norms will denoted by $\| \cdot \|$. Also, $B_p$ (resp. $B_{s,p}$) denote the unit ball of $L^p(\Omega )$ (resp. $W^{s,p}(\Omega ))$.

\smallskip
We aim to prove in the present work the following theorem.

\begin{theorem}\label{mt1}
We assume that $(\mathcal{A}_0,\mathcal{C})$ is exactly observable for $\tau \geq \tau_0$, for some $\tau_0>0$. Let $0\leq q_0 \in L^\infty(\Omega )$, there is a constant $\delta>0$  so that
\begin{equation}\label{est1}
\|q-q_0\|_2\leq C\left| \ln \left(C^{-1} \|\Lambda _{q_0}-\Lambda_q\|\right)\right|^{-1/2},\quad  q\in q_0+\delta B_{1,\infty}, 
\end{equation}
and, for any $m>0$, 
\begin{align} 
& \|a\|_2\leq C\left| \ln \left(C^{-1}\| \tilde{\Lambda} _{q_0,a}-\tilde{\Lambda} _{q_0,0}\|\right)\right|^{-1/2},\quad a\in \left[\delta B_\infty \right]\cap \left[mB_{1,2}\right],\label{est2}
\\
& \|q-q_0\|_2+ \|a\|_0\leq C\left| \ln \left(C^{-1}\| \Lambda _{q,a}-\Lambda _{q_0,0}\|\right)\right|^{-1/2},\quad q\in q_0+\delta B_{1,\infty} ,\; a\in \left[\delta B_\infty \right]\cap \left[mB_{1,2}\right].\label{est3}
\end{align}
Here, $C$ is a generic constant not depending on $q$ and $a$.
\end{theorem}


Theorem \ref{mt1} gives only stability estimates at zero damping coefficient. The difficulty of stability estimates at a non zero damping coefficient is related to the fact that the operator $\mathcal{A}$ is not necessarily diagonalizable. The main reason is that, contrary to case where $a=0$, this operator is no longer skew-adjoint. We detail the stability estimate at a non zero damping coefficient in a separate section.

\medskip
The problem of determining the potential in a wave equation from the so-called Dirichlet-to-Neumann (usually abbreviated to DN) map was initiated by Rakesh and Symes \cite{RS} (see also \cite{CM} and \cite{Is}). They prove that the potential  can be recovered uniquely from the DN map provided that the length of the time interval is larger than the diameter of the space domain. The key point in their method is the construction of special solutions, called beam solutions.   A sharp uniqueness result was proved  by the so-called boundary control method. More details on this method can be found for instance in \cite {B} and \cite{KKL}. Also,  Sun \cite{Su} establishes H\"older stability estimates and, most recently, Bao and Yun \cite{BY} improve the result of \cite{Su}. Specifically, they prove a nearly Lipschitz stability estimate. An extension was obtained by Bellassoued, Choulli and Yamamoto \cite{BCY} in the case of a partial DN map by a method built on the quantification of the continuation of the solution of the wave equation from partial Cauchy data. We refer to the introduction of \cite{BCY} for a short overview of inverse problems related to the wave equation. We finally quote a very recent  paper by Bao and Zhang \cite{BZ} dealing with sensitivity analysis of an inverse problem for the wave equation with caustics.

\medskip
It is worthwhile to mention that contrary to hyperbolic inverse problems, for which the stability can be of Lipschitz, H\"older or logarithmic type, elliptic and parabolic inverse problems are always severely ill-posed. That is the corresponding stability estimates are in most cases of logarithmic type. In \cite{Al}, Alessandrini gives an example in non destructive testing showing that the logarithmic stability is the best possible.

\medskip
This text is organized as follows. We consider in Section 2 the  inverse source problem for exactly observable systems in an abstract framework. This material is necessary to establish stability estimates for the determination of the potential and the damping coefficient appearing in the IBVP \eqref{wd}. We devote Section 3  to the proof of Theorem \ref{mt1} and we give in Section 4 a sufficient condition which guarantees  that $\mathcal{A}$ is diagonalizable. The condition that $\mathcal{A}$ is diagonalizable is used in an essential way  to get a variant of Theorem \ref{mt1} at a non zero damping coefficient. We apply in Section 5 our approach to a clamped Euler-Bernoulli beam equation. The possible adaptation of our method to a heat equation is discussed in Section 6. Due to the fact that a heat equation is not exactly observable but only observable at final time, we obtain a stability estimate only when we perturb the unknown coefficient by a finite dimensional subspace.


\section{An abstract framework for the inverse source problem} \label{back}

Let $H$ be a Hilbert space  and $A :D(A)  \subset H \rightarrow H$ be the generator of continuous semigroup $T(t)$. An operator  $C \in \mathscr{B}(D(A),Y)$,  $Y$ is a Hilbert space which is identified with its dual space, is called an admissible observation  for $T(t)$ if for some (and hence for all) $\tau >0$, the operator $\Psi \in \mathscr{B}(D(A),L^2((0,\tau ),Y))$ given by
\[
(\Psi  x)(t)=CT(t)x,\;\; t\in [0,\tau ],\;\; x\in D(A),
\]
has a bounded extension to $H$.

\smallskip
We  introduce the notion of exact observability for the system
\begin{align}\label{2.1}
&z'(t)=Az(t),\;\; z(0)=x,
\\
&y(t)=Cz(t),\label{2.2}
\end{align}
where $C$ is an admissible observation  for $T(t)$. Following the usual definition, the pair $(A,C)$ is said exactly observable at time $\tau >0$ if there is a constant $\kappa $ such that the solution $(z,y)$ of \eqref{2.1} and \eqref{2.2} satisfies 
\[
\int_0^\tau \|y(t)\|_Y^2dt\geq \kappa ^2 \|x\|_H^2,\;\; x\in D(A).
\]
Or equivalently
\begin{equation}\label{2.3}
\int_0^\tau \|(\Psi  x)(t)\|_Y^2dt\geq \kappa ^2 \|x\|_H^2,\;\; x\in D(A).
\end{equation}

\smallskip
We consider the Cauchy problem
\begin{equation}\label{2.4}
z'(t)=Az(t)+\lambda (t)x,\;\; z(0)=0,
\end{equation}
and we set
\begin{equation}\label{2.5}
y(t)=Cz(t),\;\; t\in [0,\tau ].
\end{equation}

By Duhamel's formula, we have
\begin{equation}\label{2.6}
y(t)=\int_0^t \lambda (t-s)CT(s)xds=\int_0^t\lambda (t-s)(\Psi  x)(s)ds.
\end{equation}

Let
\[
H^1_\ell ((0,\tau), Y) = \left\{u \in H^1((0,\tau), Y); \; u(0) = 0 \right\}.
\]

We define the operator $S:L^2((0,\tau), Y)\longrightarrow H^1_\ell ((0,\tau ) ,Y)$ by
\begin{equation}\label{2.7}
(Sh)(t)=\int_0^t\lambda (t-s)h(s)ds.
\end{equation}

If $E =S\Psi$, then \eqref{2.6} takes the form
\[
y(t)=(E x)(t).
\]

\begin{theorem}\label{theorem2.1}
We assume that $(A,C)$ is exactly observable for $\tau \geq \tau_0$, for some $\tau_0>0$. Let $\lambda \in H^1((0,T))$ satisfies $\lambda (0)\ne 0$. Then $E$ is one-to-one from $H$ onto $H^1_\ell ((0,\tau), Y)$ and
\begin{equation}\label{2.8}
\frac{\kappa |\lambda (0)|}{\sqrt{2}}e^{-\tau \frac{\|\lambda '\|^2_{L^2((0,\tau))}}{|\lambda (0)|^2 }}\|x\|_H\leq  \|Ex\|_{H^1_\ell ((0,\tau), Y)},\;\; x\in H.
\end{equation}
\end{theorem}

\begin{proof}
First, taking the derivative with respect to $t$ of each side of the integral equation
\[
\int_0^t \lambda (t-s)\varphi (s) ds=\psi (t),
\]
we get a Volterra equation of second kind
\[
\lambda (0)\varphi (t) +\int_0^t\lambda '(t-s)\varphi (s)ds=\psi '(t).
\]
Mimicking the proof of  \cite[Theorem 2, page 33]{Ho}, we obtain that this integral equation has a unique solution $\varphi \in L^2((0,\tau ) ,Y)$ and
\begin{align*}
\|\varphi \|_{L^2((0,\tau ),Y)}&\leq C \|\psi '\|_{L^2((0,\tau ),Y)}\\ &\leq C \|\psi \|_{H^1_\ell ((0,\tau ),Y)}.
\end{align*}
Here $C =C(\lambda )$ is a constant. 

\smallskip
Next, we estimate the constant $C$ above. From the elementary convexity inequality $(a+b)^2\leq 2(a^2+b^2)$, we deduce
\[
\| |\lambda (0)|\varphi (t)\|_Y^2\leq 2\left( \int_0^t\frac{|\lambda '(t-s)}{|\lambda (0)|}\left[|\lambda (0)|\| \varphi (s)\|_Y\right]ds \right)^2+2\|\psi '(t)\|_Y^2.
\]
Thus,
\[
|\lambda (0)|^2\| \varphi (t)\|_Y^2\leq 2\frac{\|\lambda '\|_{L^2((0,\tau))}^2}{|\lambda (0)|^2}\int_0^t|\varphi (0)|^2\| \varphi (s)\|_Y^2ds +2\|\psi '(t)\|_Y^2
\]
by the Cauchy-Schwarz's inequality. Therefore, using  Gronwall's lemma, we obtain in a straightforward manner that
\[
\| \varphi \|_{L^2((0,\tau ),Y)}\leq \frac{\sqrt{2}}{|\lambda (0)|}e^{\tau \frac{\|\lambda '\|_{L^2((0,\tau))}^2}{|\lambda (0)|^2}}\|\psi '\|_{L^2((0,\tau ),Y)}
\]
and then
\[
\| \varphi \|_{L^2((0,\tau ),Y)}\leq \frac{\sqrt{2}}{|\lambda (0)|}e^{\tau\frac{\|\lambda '\|_{L^2((0,\tau))}^2}{|\lambda (0)|^2}}\|S\varphi \|_{H^1_\ell ((0,\tau ),Y)}.
\]
In light of \eqref{2.3}, we end up getting
\[
\|Ex\|_{H^1_\ell ((0,\tau), Y)}\geq \frac{\kappa |\lambda (0)|}{\sqrt{2}}e^{-\tau \frac{\|\lambda '\|^2_{L^2((0,\tau))}}{|\lambda (0)|^2 }} \|x\|_H.
\]
\end{proof}

\medskip
We shall need a variant of Theorem \ref{theorem2.1}. If $(A,C)$ is as in Theorem \ref{theorem2.1}, then it follows from \cite[Proposition 6.3.3, page 189]{tucsnakweiss} that there is $\delta >0$ such that for any $P\in \mathscr{B}(H)$ satisfying $\|P\|\leq \delta$, $(A+P,C)$ is exactly observable with $\kappa  (P+A)\geq \kappa/2$.

\smallskip
We define $E ^P$ similarly to $E$ by replacing  $A$ by $A+P$. 

\begin{theorem}\label{theorem2.2}
We assume that $(A,C)$ is exactly observable for $\tau \geq \tau_0$, for some $\tau_0>0$. Let $\lambda \in H^1((0,T))$ satisfies $\lambda (0)\ne 0$. There is $\delta >0$ such that, for any $P\in \mathscr{B}(H)$ satisfying $\|P\|\leq \delta$, $E^P$ is one-to-one from $H$ onto $H^1_\ell ((0,\tau), Y)$ and
\begin{equation}\label{2.10}
\frac{\kappa |\lambda (0)|}{2\sqrt{2}}e^{-\tau \frac{\|\lambda '\|^2_{L^2((0,\tau))}}{|\lambda (0)|^2 }}\|x\|_H\leq  \|E ^Px\|_{H^1_\ell ((0,\tau), Y)},\;\; x\in H.
\end{equation}
\end{theorem}

\medskip
We now apply the preceding theorem to the following IBVP for the wave equation
\begin{equation}\label{2.9}
\left\{
\begin{array}{lll}
 \partial _t^2 u - \Delta u + q(x)u + a(x) \partial_t u = \lambda (t)f(x) \;\;  &\mbox{in}\;   Q,
 \\
u = 0 &\mbox{on}\;  \Sigma, 
\\
u(\cdot ,0) = 0,\; \partial_t u (\cdot ,0) = 0.
\end{array}
\right.
\end{equation}

We recall that
\[ 
\mathcal{A}_0=\left( 
\begin{array}{cc}
0 & I  \\
\Delta  & 0  \\
 \end{array} 
 \right),\;\; D(\mathcal{A}_0)=\left[H^2(\Omega )\cap H_0^1(\Omega )\right]\times H_0^1(\Omega )
 \]
 and $\mathcal{A}=\mathcal{A}_{q,a}=\mathcal{A}_0+\mathcal{B}_{q,a}$ with $D(\mathcal{A})=D(\mathcal{A}_0)$, where
 \[ 
\mathcal{B}_{q,a}=\left( 
\begin{array}{cc}
0 & 0 \\
-q & -a  \\
 \end{array} 
 \right).
 \]
Also
\[
\mathcal{C}: D(\mathcal{A}_0)\rightarrow L^2(\Gamma ): \left(\begin{array}{cc}\varphi \\ \psi \end{array}\right)\longrightarrow \partial _\nu\varphi .
\]

We fix $q_0,a_0\in L^\infty (\Omega )$ and we assume that $(\mathcal{A}_{q_0,a_0},\mathcal{C})$ is exactly observable with constant $\kappa $. This is the case when $\Gamma \supset\Gamma _0 =\{x\in \partial \Omega ;\; \nu (x)\cdot (x-x_0)> 0\}$ (see for instance \cite[Theorem 1.2, page 141]{FI}\footnote{We note that from the proof of this theorem it is not possible to extract the dependance of $\kappa $ on $q_0$ and $a_0$.}).

\begin{corollary}\label{corollary2.1}
There is $\delta >0$ such that, for any $q\in q_0+\delta B_{1,\infty}$ and $a\in a_0+\delta B_\infty $, we have  
\[
\|f\|_2\leq \frac{2\sqrt{2}}{\kappa |\lambda (0)|}e^{\tau \frac{\|\lambda '\|_{L^2((0,\tau))}^2}{|\lambda (0)|^2}}\| \partial _\nu u_f\|_{H^1((0,\tau ),L^2(\Gamma ))},
\]
where $u_f$ is the solution of the IBVP \eqref{2.9}.
\end{corollary}

This is nothing else but a Lipschitz stability estimate for the inverse problem of determining the source term $f$ from the boundary data $\partial _\nu u _f{_{|\Upsilon}}$, when $\lambda$ is supposed to be known.


\section{Proof of Theorem \ref{mt1}} 

\begin{proof}[Proof of Theorem \ref{mt1}]
Let $(\lambda _k)$ and $(\phi _k)$ be respectively the sequence of Dirichlet eigenvalues of $-\Delta +q_0$, counted according to their multiplicity, and the corresponding eigenvectors. We assume that the sequence $(\phi _k)$ forms an orthonormal basis of $L^2(\Omega )$.

\smallskip
We recall that according to the min-max principle, the following two-sided estimates hold
\begin{equation}\label{ineq1}
c^{-1}k^{2/n}\leq \lambda _k\leq ck^{2/n}.
\end{equation}
Here, the constant $c>1$ depends only on $\Omega$ and $q_0$.

\smallskip
Let $u_q$ be the solution of the IBVP \eqref{wd} corresponding to $q$, $a=0$, $u_0=\phi _k$ and $u_1=0$. Taking into account that $u_{q_0}=\cos (t\sqrt{\lambda _k})\phi _k$ is the solution of the IBVP \eqref{wd} corresponding to $q=q_0$, $a=0$, $u_0=\phi _k$ and $u_1=0$, we see that $u=u_q -u_{q_0}$ is the solution of the IBVP
\begin{equation}\label{e2}
\left\{
\begin{array}{lll}
\partial _t^2u - \Delta u + qu = -(q-q_0) \cos(t \sqrt{\lambda_k}) \phi_k \;\; &\mbox{in}\;  Q, 
\\
u= 0 &\mbox{on}\; \Sigma , 
\\
u(\cdot ,0) = 0,\; \partial _tu (\cdot ,0) = 0.
\end{array}
\right.
\end{equation}

In the remaining part of this proof, $C$ is a generic constant independent on $k$.

\smallskip
Let $\delta$ be as in Corollary \ref{corollary2.1}. If $q\in q_0+\delta B_{1,\infty}$, we get by applying Corollary \ref{corollary2.1}
\[
\|(q-q_0) \phi_k\|_2\leq Ce^{C\lambda _k}\|\partial _\nu u\|_{H^1 ((0,\tau),L^2 (\Gamma))}.
\]
Since $|(q-q_0,\phi_k)|\leq |\Omega |^{1/2}\|(q-q_0) \phi_k\|_{L^2(\Omega )}$ by Cauchy-Schwarz's inequality, the last inequality entails
\[
|(q-q_0,\phi_k)|\leq Ce^{C\lambda _k}\|\partial _\nu u\|_{H^1 ((0,\tau),L^2 (\Gamma))}.
\]
But, $\partial _\nu u=(\Lambda _{q_0}-\Lambda_q)\phi _k$. Therefore 
\begin{equation}\label{e1}
|(q-q_0,\phi_k)|^2\leq Ce^{C\lambda _k}\| \Lambda _{q_0}-\Lambda_q\|^2.
\end{equation}

Let $\lambda \geq \lambda _1$ and $N=N(\lambda )$ be the smallest integer so that $\lambda _N\leq \lambda <\lambda _{N+1}$. Then
\begin{align*}
\|q-q_0\|_2^2 &=\sum_k |(q-q_0,\phi_k)|^2
\\
&=\sum_{k\leq N} |(q-q_0,\phi_k)|^2+\sum_{k>N}|(q-q_0,\phi_k)|^2
\\
&\leq \sum_{k\leq N} |(q-q_0,\phi_k)|^2+\frac{1}{\lambda }\sum_{k>N}\lambda _k |(q-q_0,\phi_k)|^2
\\
&\leq \sum_{k\leq N} |(q-q_0,\phi_k)|^2+\frac{C\delta^2}{\lambda }.
\end{align*}
Here we used the fact that $\left( \sum_{k\geq 1} (1+\lambda _k)(\cdot ,\varphi _k)^2\right)^{1/2}$ defines an equivalent norm on $H^1(\Omega )$.

In light of \eqref{e1}, we get
\[
\|q-q_0\|_2^2\leq CNe^{C\lambda }\| \Lambda _{q_0}-\Lambda_q\|^2+\frac{C\delta^2}{\lambda }.
\]
By \eqref{ineq1}, $N\leq C\lambda ^{n/2}$. Hence
\[
\|q-q_0\|_2^2\leq Ce^{C\lambda }\| \Lambda _{q_0}-\Lambda_q\|^2+\frac{C\delta^2}{\lambda }.
\]
Minimizing with respect to $\lambda$, we obtain that there is $\delta _0>0$ such that if $\| \Lambda _{q_0}-\Lambda_q\|\leq \delta _0$, then
\[
\|q-q_0\|_2\leq C\left| \ln (C^{-1} \|\Lambda _{q_0}-\Lambda_q\|)\right|^{-1/2}.
\]
Estimate \eqref{est1} follows then from the continuity of the mapping \[ q\in L^\infty (\Omega )\rightarrow \Lambda _q \in  \mathscr{B}(H_0^1(\Omega )\cap H^2(\Omega ),H^1((0,\tau ),L^2(\Gamma )).\]

\medskip
We proceed similarly for proving \eqref{est2}. In the actual case we have to replace the previous $u_{q_0}$ by $u_{q_0}=\lambda _k^{-1}\sin (t\sqrt{\lambda _k})\phi _k$, corresponding to the initial conditions $u_0=0$ and $u_1=\phi _k$. Therefore, we have in place of \eqref{e2}
\begin{equation}\label{e3}
\left\{
\begin{array}{lll}
\partial _t^2u - \Delta u + a \partial_t u = -a \cos(t \sqrt{\lambda_k}) \phi_k \;\; &\mbox{in}\;  Q, 
\\
u= 0 &\mbox{on}\; \Sigma , 
\\
u(\cdot ,0) = 0,\; \partial _tu (\cdot ,0) = 0.
\end{array}
\right.
\end{equation}
We continue as in the preceding case by establishing  the estimate
\[
|(a,\phi_k)|^2\leq Ce^{C\lambda _k}\| \widetilde{\Lambda} _{q_0,a}-\widetilde{\Lambda} _{q_0,0}\|
\]
and we complete the proof of \eqref{est2} as above.

\medskip
We end the proof by showing how we proceed for proving \eqref{est3}. Taking into account that the solution corresponding to $q=q_0$, $a=0$, $u_0=\phi _k$ and $u_1=i\lambda _k\phi _k$ is $u_{q_0}=e^{i\sqrt{\lambda _k}t}\phi _k$, then in place of \eqref{e2} we have the following IBVP
\begin{equation}\label{e4}
\left\{
\begin{array}{lll}
\partial _t^2u - \Delta u + qu +a\partial _tu = -[(q-q_0)+ i\sqrt{\lambda _k}a]e^{i\sqrt{\lambda _k}t} \phi_k, \;\; &\mbox{in}\;  Q, 
\\
u= 0 \;\; &\mbox{on}\; \Sigma , 
\\
u(\cdot ,0) = 0,\; \partial _tu (\cdot ,0) = 0.

\end{array}
\right.
\end{equation}

We can argue one more time as in the proof of \eqref{est1}. We find
 \[
 |(\varphi ,q-q_0)+i\sqrt{\lambda _k}(\varphi ,a)|^2\leq Ce^{C\lambda _k}\|\Lambda_{q,a}-\Lambda_{q_0,0}\|^2,
 \] 
 entailing
 \begin{align*}
 &|(\varphi ,q-q_0)|^2\leq Ce^{C\lambda _k}\|\Lambda_{q,a}-\Lambda_{q_0,0}\|^2,
 \\
 &|(\varphi ,a)|^2\leq Ce^{C\lambda _k}\|\Lambda_{q,a}-\Lambda_{q_0,0}\|^2.
 \end{align*}
We end up getting \eqref{est3} by mimicking the rest of the proof of  estimate \eqref{est1}.
\end{proof}


\section{Stability  around a non zero damping coefficient}

We limit ourselves to the one dimensional case and, for sake of simplicity, we take $q$ identically equal to zero. But  the analysis we carry out in the present section is still applicable for any non negative bounded potential.

\smallskip
We assume in the present section that $\Omega =(0,\pi )$. We introduced in the first section the unbounded operators, defined on $\mathcal{H}=H_0^1(\Omega )\times L^2(\Omega )$,
\[ 
\mathcal{A}_0=\left( 
\begin{array}{cc}
0 & I  \\
\frac{d^2}{d x^2}  & 0  \\
 \end{array} 
 \right),\;\; D(\mathcal{A}_0)=\left[H^2(\Omega )\cap H_0^1(\Omega )\right]\times H_0^1(\Omega ):=\mathcal{H}_1
 \]
 and $\mathcal{A}_a=\mathcal{A}_0+\mathcal{B}_a$ with $D(\mathcal{A})=D(\mathcal{A}_0)$, where
 \[ 
\mathcal{B}_a=\left( 
\begin{array}{cc}
0 & 0 \\
0 & -a  \\
 \end{array} 
 \right).
 \]

From \cite[Proposition 6.2.1, page 180]{tucsnakweiss}, $(\mathcal{A}_0,\mathcal{C})$ is exactly observable for any $\tau \geq 2\pi$ when
\[
\mathcal{C}: \left(\begin{array}{cc}\varphi \\ \psi \end{array}\right)\in D(\mathcal{A}_0)\longrightarrow \frac{d\varphi}{dx}(0).
\]
On the other hand, it follows from \cite[Proposition 3.7.7, page 101]{tucsnakweiss} that the skew-adjoint operator $\mathcal{A}_0$ is diagonalizable with eigenvalues $\lambda _k=ik$, $k\in \mathbb{Z}^\ast$, corresponding to the orthonormal basis $(g_k)$, where
\[
g_k=\frac{1}{\sqrt{2}}\left( \begin{array}{cc} \frac{f_k}{ik}\\ f_k\end{array}\right),\;\; k\in \mathbb{Z}^\ast ,
\]
where $(f_k)_{k\in \mathbb{N}}$ is an orthonormal basis of $L^2(\Omega )$ consisting of eigenfunctions of the unbounded operator $A_0=\frac{d^2}{dx^2}$ under Dirichlet boundary condition and $f_{-k}=-f_k$, $k\in \mathbb{N}^\ast$.

\smallskip
Let $\mathcal{H}_{\pm}$ be the closure in $\mathcal{H}$ of ${\rm span}\{ g_{\pm k};\; k\in \mathbb{N}^\ast\}$. Clearly, $\mathcal{H}=\mathcal{H}_+\oplus \mathcal{H}_-$ and $\mathcal{H}_\pm$ is invariant under $\mathcal{A}_0$. Let then $\mathcal{A}_0^\pm:\mathcal{H}_\pm \rightarrow \mathcal{H}_\pm$ be the unbounded operator given by $\mathcal{A}_0^\pm=\mathcal{A}_0{_{|\mathcal{H}_\pm}}$ and
\[
D(\mathcal{A}_0^\pm )=\{u\in \mathcal{H}_\pm ;\; \sum_{k\in \mathbb{N}^\ast}k^2|\langle u,g _{\pm k}\rangle|^2<\infty \}.
\]
Here $\langle \cdot ,\cdot \rangle$ is the scalar product in $\mathcal{H}$.

\smallskip
Let $\mathcal{A}_{a_0}^\pm=\mathcal{A}_0^\pm + \mathcal{B}_{a_0}$ and set
\[
\varrho=\sum_{k\geq 1}\frac{1}{(2k+1)^2}\;\; \mbox{and}\;\; \alpha =\frac{1}{2\sqrt{2(1+\varrho )}}.
\]

In light of \cite[Theorem 2 and Lemma 10]{Shk}, we get

\begin{theorem}\label{t4.1}
Under the assumption
\[
\rho :=\|a_0\|_\infty <\alpha ,
\]
the spectrum of $\pm \mathcal{A}_{a_0}^\pm$ consists in a sequence $(i\mu _k^\pm)$ such that, for any $\delta \in (0,1-\rho^2/\alpha ^2)$, there is an integer $\widetilde{k}$ such that
\[
|i\mu_k^\pm -ik|\leq \overline{\alpha}=\overline{\alpha}(a_0):=\frac{\rho }{\sqrt{4\rho^2+\delta}},\;\; k\geq \widetilde{k}.
\]
In addition, $\mathcal{H}_\pm$ admits a Riesz basis $(\phi _k^\pm)=\left(\left( \begin{array}{c}\varphi _k^\pm \\ i\mu _k^\pm \varphi _k^\pm \end{array}\right)\right)_{k\in \mathbb{N}^\ast}$, each $\phi_k^\pm$ is an eigenfunction corresponding to $i\mu _k^\pm$. 
\end{theorem}

We denote by $(\widetilde{\phi}_k^\pm)$ the Riesz basis biorthogonal to $(\phi _k^\pm)$ and define the sequence $(\phi _k)_{k\in \mathbb{Z}^\ast}$ (resp. $(\widetilde{\phi}_k)_{k\in \mathbb{Z}^\ast}$) as follows $\phi_{-k}=-\phi _k^-$ and $\phi_k=\phi _k^+$ (resp. $\widetilde{\phi}_{-k}=-\widetilde{\phi }_k^-$ and $\widetilde{\phi}_k=\widetilde{\phi} _k^+$), $k\in \mathbb{N}^\ast$. Set also $\mu_{-k}=-\mu_k^-$ and $\mu_k=\mu_k^+$, $k\in \mathbb{N}^\ast$. Therefore, $\mathcal{A}_{a_0}\phi_k=i\mu_k\phi_k$, $k\in \mathbb{Z}^\ast$, and, for any $u\in \mathcal{H}$,
\[
u=\sum_{k\in \mathbb{Z}}\langle u,\widetilde{\phi}_k\rangle \phi_k =\sum_{k\in \mathbb{Z}}\langle u,\phi_k\rangle \widetilde{\phi}_k.
\]
Additionally,
\begin{equation}\label{4.1}
\alpha \|u\|^2_{\mathcal{H}}\leq \sum_{k\in \mathbb{Z}^\ast}|\langle u,\widetilde{\phi}_k\rangle|^2,\; \sum_{k\in \mathbb{Z}^\ast}|\langle u,\phi_k\rangle|^2 \leq \beta \|u\|^2_{\mathcal{H}},
\end{equation}
where the constants $\alpha$ and $\beta$ do not depend on $u$ (see for instance \cite[Lemma 252, page 37]{tucsnakweiss}).

\medskip
We pick $a_0$ as in the preceding theorem. Then it is straightforward to check that $u_{a_0}=e^{i\mu _kt}\varphi_k$, $k\in \mathbb{Z}^\ast$, is the solution of the IBVP \eqref{wd} with $q=0$, $a=a_0$, $\left( \begin{array}{c}u_0 \\ u_1 \end{array}\right)=\phi_k$. If $u_a$ is the solution of the IBVP \eqref{wd}, then $u=u_a-u_{a_0}$ is the solution of the IBVP

\begin{equation}\label{4.2}
\left\{
\begin{array}{lll}
 \partial _t^2 u - \Delta u  + a(x) \partial_t u = (a_0-a)i\mu_ke^{i\mu _kt}\varphi _k \;\; &\mbox{in}\;   Q, 
 \\
u = 0 &\mbox{on}\;  \Sigma, 
\\
u(\cdot ,0) =0 ,\; \partial_t u (\cdot ,0) = 0.
\end{array}
\right.
\end{equation}

We fixe $\delta$ as in the statement of Theorem \ref{t4.1}. Then, for some integer $\widetilde{k}$,
\begin{align*}
&\left| e^{i\mu _kt}\right|\leq e^{|i\mu _k-i|k||t}\left|e^{i|k|t}\right|\leq e^{\overline{\alpha}\tau},\;\; |k|\geq \widetilde{k},
\\
&|\mu _k|\leq |k|+\overline{\alpha},\;\; |k|\geq \widetilde{k}.
\end{align*}

These estimates at hand, we can proceed as in the previous section to get, where $\psi _k=i\mu _k\varphi _k$,
\begin{equation}\label{4.3}
 |(a-a_0, \psi _k)|^2=\left| \left\langle \left( \begin{array}{c} 0\\ a-a_0 \end{array}\right),\phi_k \right\rangle  \right|^2\leq Ce^{Ck^2}\|\Lambda_a-\Lambda_{a_0}\|^2.
 \end{equation}
 
It follows from \eqref{4.1},
\begin{equation}\label{4.4}
\alpha \|a-a_0\|^2_2=\alpha \left\| \left( \begin{array}{c} 0\\ a-a_0 \end{array}\right)\right\|_{\mathcal{H}}^2\leq \sum_{|k|\geq 1}
\left| \left\langle \left( \begin{array}{c} 0\\ a-a_0 \end{array}\right),\phi_k \right\rangle  \right|^2.
\end{equation}

In light of \eqref{4.3} and \eqref{4.4}, we have
\begin{align}
\alpha \|a-a_0\|^2_2&\leq CNe^{C\lambda }\|\Lambda_a-\Lambda_{a_0}\|^2+\frac{1}{\lambda }\sum_{|k|>N}k^2 |(a-a_0, \psi _k)|^2\nonumber
\\
&\leq CNe^{C\lambda }\|\Lambda_a-\Lambda_{a_0}\|^2+\frac{1}{\lambda }\sum_{|k|\geq 1}k^2 |(a-a_0, \psi _k)|^2.\label{4.5}
\end{align}
Here $\lambda \geq \lambda _1$ and $N=N(\lambda )$ be the smallest integer satisfying $N^2\leq \lambda <(N+1)^2$.

\smallskip
We note that we cannot pursue the proof similarly to that of \eqref{est1} because $(\psi _k)$ is not necessarily an orthonormal basis of $L^2(\Omega )$. So instead of the boundedness of $a-a_0$ in  $H^1(\Omega )$, we make the assumption, where $m>0$ is fixed,
\begin{equation}\label{4.6}
\sum_{|k|\geq 1}k^2 |(a-a_0, \psi _k)|^2\leq m.
\end{equation}

Under the assumption \eqref{4.6}, \eqref{4.5} entails
\[
\alpha \|a-a_0\|^2_2\leq Ce^{C\lambda }\|\widetilde{\Lambda}_a-\widetilde{\Lambda}_{a_0}\|^2+\frac{m}{\lambda }.
\]
where $\widetilde{\Lambda}_a=\widetilde{\Lambda}_{0,a}$ and $\widetilde{\Lambda}_{a_0}=\widetilde{\Lambda}_{0,a_0}$.

\smallskip
The same minimization argument used in the proof  of \eqref{est1} (see Section 3) allows us to prove the following theorem.

\begin{theorem}\label{theorem4.2}
There exist two constants $C>0$ and $\delta >0$ so that
\[
\|a-a_0\|_2\leq C\left| \ln \left(C^{-1}\| \widetilde{\Lambda} _a-\widetilde{\Lambda} _{a_0}\|\right)\right|^{-1/2},\quad a \in a_0+\delta B_\infty \; \mbox{and \eqref{4.6} holds}.
\]
\end{theorem}

\begin{remark}\label{remark4.1}
Let us explain briefly why the result of this section can not be extended to a higher dimensional case. The main reason is that, even for simple geometries,  the eigenvalues of the unperturbed  operators $\mathcal{A}_0^\pm$ do not satisfy a gap condition which is the main assumption in \cite[Theorem 2]{Shk}. If $(\rho _k)$, $\rho_k=k$, is the sequence of eigenvalues of $\pm \mathcal{A}_0^\pm $, we used in an essential way that
\[
\rho_{k+1} -\rho_k =1.
\]
When $\Omega = (0,a)\times (0,b)$, the eigenvalues operator $\mathcal{A}_0^+$ consist in the sequence $\left( \pi ^2\left( \frac{k^2}{a^2}+\frac{\ell ^2}{b^2}\right) \right)_{k,\ell \in \mathbb{N}^\ast}$. These eigenvalues are simple when $\frac{a^2}{b^2}\not \in \mathbb{Q}$ but can condensate in finite interval and therefore they don't satisfy a gap condition like in the one dimensional case.
\end{remark}


\section{An application to clamped Euler-Bernoulli beam}

For the same reason as in the preceding section, we limit our analysis to the one dimensional case. So we let $\Omega =(0,1)$.

\smallskip
We introduce the following spaces
\begin{align*}
& H_0=L^2(0,1),
\\
&H_{1/2}=H_0^2(\Omega ),
\\
&H_1=H^4(0,1)\cap H_0^2(\Omega ).
\end{align*}

The natural norm of $H_s$ will denoted by $\|\cdot \|_s$, $s\in \{0,1/2,1\}$.

\smallskip
On $\mathcal{H}=H_{1/2}\times H_0$, we introduce the unbounded operator $\mathcal{A}$ given by
\[
\mathcal{A}=\left(
\begin{array}{cc} 0&I\\- \frac{d^4}{dx^4}&0\end{array} \right),\;\; D(\mathcal{A})=H_1\times H_{1/2}:=\mathcal{H}_1.
\]

We consider a torque observation at an end point. We define then $C:\mathcal{H}_1\rightarrow \mathbb{C}$ by
\[
C\left(\begin{array}{cc}\varphi \\\psi \end{array}\right)= \frac{d^2\varphi}{dx^2}(0).
\]

\smallskip
We are concerned with following IBVP for the clamped Euler-Bernoulli beam equation
\begin{equation}\label{eb0}
\left\{
\begin{array}{lll}
 \partial _t^2 u + \partial _x^4 u = 0 \;\; &\mbox{in}\;   Q, 
\\
u(0,\cdot ) = u(1,\cdot )= 0 \;\; &\mbox{on}\;  (0,\tau ), 
\\
\partial _xu(0,\cdot ) = \partial _xu(1,\cdot )= 0 \;\; &\mbox{on}\;  (0,\tau ), 
\\
u(\cdot ,0) = u_0,\; \partial_t u (\cdot ,0) = u_1.
\end{array}
\right.
\end{equation}

From \cite[Proposition 3.7.6, page 100]{tucsnakweiss}, $\mathcal{A}$ is skew-adjoint and therefore it generates a unitary group on $\mathcal{H}$. Consequently, for any $\left(
\begin{array}{c} u_0\\u_1\end{array} \right)\in \mathcal{H}_1$ the IBVP \eqref{eb0} has a unique solution $u$ so that $(u,u')\in C([0,\tau ],\mathcal{H}_1)\cap C^1([0,\tau ],\mathcal{H})$. Moreover, by \cite[Proposition 6.10.1, page 270]{tucsnakweiss}, $(\mathcal{A},\mathcal{C})$ is exactly observable for any $\tau >0$ and there is a constant $\kappa$ such that
\begin{equation}\label{p1}
\kappa ^2(\|u_0\|_{1/2}^2+\|u_1\|_0^2)\leq \| \partial _x^2u(0,\cdot )\|^2_{L^2((0,\tau ))}.
\end{equation}
Here the constant $\kappa$ is independent on $u_0$ and $u_1$.

\smallskip
Let $\mathcal{B}_a$ be the operator, where $a=a(x)$,
\[
\mathcal{B}_a=\left(
\begin{array}{cc} 0&0\\0& - a\end{array} \right).
\]
This operator is bounded on $\mathcal{H}$ whenever $a\in L^\infty (\Omega )$. Therefore, bearing in mind that $\mathcal{A}+\mathcal{B}_a$ generates a continuous semigroup, the IBVP 

\begin{equation}\label{eb1}
\left\{
\begin{array}{lll}
 \partial _t^2 u + \partial _x^4 u +a(x)\partial _tu= 0 \;\; &\mbox{in}\;   Q, 
\\
u(0,\cdot ) = u(1,\cdot )= 0 \;\; &\mbox{on}\;  (0,\tau ), 
\\
\partial _xu(0,\cdot ) = \partial _xu(1,\cdot )= 0 \;\; &\mbox{on}\;  (0,\tau ), 
\\
u(\cdot ,0) = u_0,\; \partial_t u (\cdot ,0) = u_1
\end{array}
\right.
\end{equation}
has a unique solution $u=u_a(u_0,u_1)$ satisfying $(u,u')\in C([0,T],\mathcal{H}_1)\cap C^1([0,T],\mathcal{H})$, for any $\left( \begin{array}{c} u_0\\u_1\end{array} \right)\in \mathcal{H}_1$. Moreover, the same perturbation argument used in the proof of Theorem \ref{theorem2.2} enables us to show that $(\mathcal{A}+\mathcal{B}_a,\mathcal{C})$ is exactly observable with constant $\widetilde{\kappa}^2\geq \kappa ^2/2$ provided the norm of  the operator $\mathcal{B}_a$ is sufficiently small. That is, there is $\delta >0$ such that for any $\mathcal{B}_a\in \mathscr{B}(\mathcal{H})$ with $\| \mathcal{B}_a\|\leq \delta$, we have 
\begin{equation}\label{p2}
(1/2)\kappa ^2(\|u_0\|_{1/2}^2+\|u_1\|_0^2)\leq \| \partial _x^2u(0,\cdot )\|^2_{L^2((0,\tau ))}.
\end{equation}

In light of \cite[Lemma 6.10.2, page 218]{tucsnakweiss}, the spectrum of $\mathcal{A}$ consists in a sequence of simple eigenvalues $(i\rho _k)_{k\in \mathbb{Z}^\ast}$, where
\[
\rho _k=\pi^2\left( k-\frac{1}{2} \right)^2+a_k, \;\; k\in \mathbb{N}^\ast ,
\]
$(a_k)$ a sequence converging exponentially to $0$, and $\rho_{-k}=-\rho_k$, $k\in \mathbb{N}^\ast$.

\smallskip
Let $A_0$ be the unbounded operator on $L^2(\Omega )$ defined by $A_0=\frac{d^4}{dx^4}$ and $D(A_0)=H^4(\Omega )\cap H_0^2(\Omega )$. Then $A_0$ is diagonalizable with eigenvalues $(\rho_k^2)_{k\in \mathbb{N}^\ast}$. Let $(f_k)_{k\in \mathbb{N}^\ast}$ be a basis of eigenfunctions, each $f_n$ is an eigenfunction corresponding to $\rho _k^2$. Let 
\[
g_k=\frac{1}{\sqrt{2}}\left(\begin{array}{cc} \frac{f_k}{i\rho _k}\\ f_k\end{array}\right),\;\; \textrm{and}\;\; g_{-k}=-g_k,\;\; k\in \mathbb{N}^\ast .
\]
With the help of \cite[Lemma 3.7.7, page 101]{tucsnakweiss}, we get that $(g_k)_{k\in \mathbb{Z}^\ast}$ is an orthonormal basis of $\mathcal{A}_0$.

\smallskip
Define $\mathcal{H}_\pm$ as the closure of $\textrm{span}\{ g_{\pm k};\; k\in \mathbb{N}^\ast \}$. Then $\mathcal{H}=\mathcal{H}_-\oplus \mathcal{H}_+$ and $\mathcal{H}_\pm$ is invariant under $\mathcal{A}_0$. We consider $\mathcal{A}_0^\pm:\mathcal{H}_\pm \rightarrow \mathcal{H}_\pm$ the unbounded operator given by $\mathcal{A}_0^\pm=\mathcal{A}_0{_{|\mathcal{H}_\pm}}$ and
\[
D(\mathcal{A}_0^\pm )=\{u\in \mathcal{H}_\pm ;\; \sum_{k\in \mathbb{N}^\ast}k^2|\langle u,g _{\pm k}\rangle|^2<\infty \},
\]
where $\langle \cdot ,\cdot \rangle$ is the scalar product in $\mathcal{H}$, and we set $\mathcal{A}_{a_0}^\pm =\mathcal{A}_{a_0}^\pm +\mathcal{B}_{a_0}$.

\smallskip
Since $\rho_{k+1}-\rho_k\rightarrow +\infty$ as $k\rightarrow +\infty$, $(\rho _k)_{k\in \mathbb{N}^\ast}$ satisfies the a gap condition. Precisely, there exists $d>0$ so that 
\[
\rho_{k+1}-\rho_k\geq d,\;\; k\in \mathbb{N}^\ast.
\]

Set 
\[
\alpha '=\frac{d}{2\sqrt{2(1+\varrho )}},
\]
where $\varrho$ is as in Section 4.

\smallskip
We have similarly to Theorem \ref{t4.1},

\begin{theorem}\label{t5.1}
Under the assumption
\[
\rho :=\|a_0\|_\infty <\alpha ',
\]
the spectrum of $\pm\mathcal{A}_{a_0}^\pm$ consists in a sequence $(i\mu _k^\pm)$ such that, for any $\delta \in (0,1-\rho^2/(\alpha ')^2)$, there is an integer $\widetilde{k}$ such that
\[
|i\mu_k^\pm -i\rho_k|\leq \overline{\alpha}=\overline{\alpha}(a_0):=\frac{\rho d}{\sqrt{4\rho^2+d^2\delta}},\;\; k\geq \widetilde{k}.
\]
In addition, $\mathcal{H}^\pm $ admits a Riesz basis $(\phi _k^\pm )=\left(\left( \begin{array}{c}\varphi _k^\pm \\ i\mu _k^\pm\varphi _k^\pm \end{array}\right)\right)$, each $\phi_k^\pm $ is an eigenfunction corresponding to $i\mu _k^\pm$. 
\end{theorem}

We define the IB operator  $\tilde{\Lambda} _a$ by
\[
 \tilde{\Lambda} _a : u_1\in H_{1/2}\longrightarrow \partial _x^2u_a(0,u_1)(0,\cdot )\in L^2(0,\tau ).
\]
One can prove in a straightforward manner that $ \tilde{\Lambda} _a$ is bounded  operator between $\mathcal{H}_1$ and $H^1((0,\tau ))$ and its norm can be uniformly bounded, with respect to $a$, by a constant, provided that the $L^\infty$-norm of $a$ is sufficiently small.

\medskip
We carry out  a similar analysis to that after Theorem \ref{t4.1} to get the following stability estimate. 

\begin{theorem}\label{theorem5.2}
Given $m>0$, there exist constants $C>0$ and $\delta >0$ so that
\[
\|a-a_0\|_0\leq C\left| \ln \left(C^{-1}\| \tilde{\Lambda} _a-\tilde{\Lambda} _{a_0}\|\right)\right|^{-1/4},
\]
if $a\in a_0+\delta B_{1,\infty}$ and 
\[
\sum_{|k|\geq 1}\lambda_k |(a-a_0, \psi _k)|^2\leq m,
\]
where $\psi _{\pm k}=i\mu _k^\pm\varphi _k^\pm$, $k\in \mathbb{N}^\ast$.
\end{theorem}

\medskip
We mention that the method used in this section and in the previous one is easily adaptable to a Schr\"odinger equation.


\section{The case of a heat equation} 

We consider the following IBVP for the heat equation
\begin{equation}\label{he}
\left\{
\begin{array}{lll}
 \partial _t u - \Delta u + q(x)u  = 0 \;\; &\mbox{in}\;   Q, 
 \\
u = 0 &\mbox{on}\;  \Sigma, 
\\
u(\cdot ,0) = u_0.
\end{array}
\right.
\end{equation}
Let $H^{2,1}(Q)=L^2((0,\tau ),H^2(\Omega ))\cap H^1((0,\tau ), L^2(\Omega ))$. From \cite[Theorem 1.43, page 27]{choulli}, for any $q\in L^\infty (\Omega )$ and $u_0\in H_0^1(\Omega )$, the IBVP has a unique solution $u_q=u_q(u_0)\in H^{2,1}(Q)$ and, for any $m>0$,
\[
\|u_q\|_{H^{2,1}(Q)}\leq C\|u_0\|_{1,2},
\]
where the constant $C=C(M)$ is independent on $q$, $\|q\|_\infty \leq m$.

\smallskip
Let $\Gamma$ be an arbitrary nonempty open subset of $\partial \Omega$ and set $\Upsilon =\Gamma \times (0,\tau )$. Using the trace theorem in \cite[page 26]{choulli}, we obtain  that the following IB mapping
\[
\Lambda _q :u_0\in H_0^1(\Omega )\longrightarrow \partial _\nu u_q(u_0)\in L^2(\Upsilon)
\]
is bounded.

\smallskip
The following lemma will be useful in the sequel. Its proof is sketched in Appendix A.

\begin{lemma}\label{lemmah1}
Let $q_0,q\in L^\infty (\Omega )$ so that $q\in q_0+W^{1,\infty}(\Omega )$. Then $\Lambda _q-\Lambda _{q_0}$ defines a bounded operator from $H_0^1(\Omega )$ into $H^1((0,\tau );L^2(\Gamma ))$. Additionally, for each $m>0$, there exits $C>0$ so that 
\[
\| \Lambda _q-\Lambda _{q_0}\|\leq C,
\]
for all  $q_0,q\in mB_\infty$. Here, $\| \Lambda _q-\Lambda _{q_0}\|$ is the norm of $\Lambda _q-\Lambda _{q_0}$ in $\mathscr{B}(H_0^1(\Omega );H^1((0,\tau );L^2(\Gamma )))$.
\end{lemma}

In the sequel $\Lambda _q-\Lambda _{q_0}$ is considered as an operator acting from $H_0^1(\Omega )$ into $H^1((0,\tau );L^2(\Gamma ))$.

\smallskip
We assume, without loss of generality, that $q\geq 0$. Indeed, substituting $u$ by $ue^{-\|q\|_\infty t}$, we  see that $q$ in \eqref{he} is changed to $q+\|q\|_\infty$. So, we fix $q_0\in L^\infty (\Omega )$ satisfying $0\leq q_0$ and we let  $0<\lambda _1<\lambda _2 \ldots \leq \lambda _k\rightarrow +\infty$ be the sequence of  eigenvalues, counted according to their multiplicity, of $-\Delta +q_0$ under Dirichlet boundary condition. An orthonormal basis consisting in the corresponding eigenfunctions is denoted by $(\varphi _k)$.

\smallskip
Let $q\in mB_\infty \cap \left(q_0+W^{1,\infty}(\Omega )\right)$. We pick a positive integer $k$. Taking into account that $u_{q_0}(\varphi_k)=e^{-\lambda _k t}\varphi _k$, we obtain that $u=u_q(\varphi _k)-u_{q_0}(\varphi_k)$ is the solution of the IBVP
\begin{equation}\label{h1}
\left\{
\begin{array}{lll}
 \partial _t u - \Delta u + q(x)u  = (q_0-q)e^{-\lambda _kt}\varphi _k \;\; &\mbox{in}\;   Q, 
 \\
u = 0 &\mbox{on}\;  \Sigma, 
\\
u(\cdot ,0) = 0.
\end{array}
\right.
\end{equation}

We set $f=(q-q_0)\varphi _k$ and $\lambda (t)=e^{-\lambda _kt}$. Therefore \eqref{h1} becomes 
\begin{equation}\label{h2}
\left\{
\begin{array}{lll}
 \partial _t u - \Delta u + q(x)u  = \lambda (t) f(x) \;\;& \mbox{in}\;   Q, 
 \\
u = 0 &\mbox{on}\;  \Sigma , 
\\
u(\cdot ,0) = 0.
\end{array}
\right.
\end{equation}

 It is straightforward to check that 
 \begin{equation}\label{h3}
 u(x,t)=\int_0^t\lambda (t-s)v(x,s),
 \end{equation}
 where $v$ is the solution of the IBVP
 \[
\left\{
\begin{array}{lll}
 \partial _t v - \Delta v + q(x)v  = 0 \;\; &\mbox{in}\;   Q, 
 \\
v = 0 &\mbox{on}\;  \Sigma , 
\\
v(\cdot ,0) = f.
\end{array}
\right.
\]

In light of the Carleman estimate in \cite[Theorem 3.4, page 165]{choulli}, we can extend \cite[Proposition 3.5, page 170]{choulli} in order to get the following final time observability inequality 
\begin{equation}\label{h4}
\|v(\cdot ,\tau )\|_{H_0^1(\Omega )}\leq C\|\partial _\nu v\|_{L^2(\Upsilon)}.
\end{equation}
Here $C$ is a constant depending on $m$ but not on $q$.

\smallskip
By the continuity of the trace operator $w\in H^{2,1}(Q)\rightarrow \partial _\nu w_{|\Upsilon}\in L^2(\Upsilon)$, we get from \eqref{h3}
\[
\partial _\nu u(x,t)_{|\Upsilon}=\int_0^t\lambda (t-s)\partial _\nu v(x,s)_{|\Upsilon}.
\]
We proceed as in the beginning of the proof of Theorem \ref{theorem2.1} to deduce the following estimate
\begin{equation}\label{h7}
\|\partial _\nu v\|_{L^2(\Upsilon )}\leq \sqrt{2}e^{\tau ^2\lambda _k^2}\|\partial _\nu u\|_{H^1((0,\tau ), L^2(\Gamma ))}.
\end{equation}

\smallskip
On the other hand
\begin{equation}\label{h5}
v(x,t)=\sum_{\ell \geq 1}e^{-\lambda _\ell t}(f,\varphi _\ell )\varphi _\ell .
\end{equation}
Hence
\[
\|v(\cdot ,\tau )\|_2^2= \sum_{\ell \geq 1}e^{-2\lambda _\ell \tau }|(f,\varphi _\ell )|^2.
\]
Arguing as in Section 3, we get, for any $\lambda \geq \lambda _1$ and $N=N(\lambda )$ satisfying $\lambda _N\leq \lambda <\lambda_{N+1}$,
\begin{equation}\label{h6}
\|f\|_2^2\leq e^{2\lambda _k \tau }\|v(\cdot ,\tau )\|_2+\frac{1}{\lambda ^2}\sum_{\ell >N}\lambda _\ell ^2|(f,\varphi _\ell )|^2.
\end{equation}

By Green's formula, we obtain
\[
\lambda _\ell (f,\varphi _\ell)=-\int_\Omega \Delta (q-q_0)\varphi_\ell \varphi_kdx +2\int_\Omega \nabla (q-q_0)\cdot \nabla \varphi_k\varphi_\ell dx +\lambda _k(f,\varphi _\ell).
\]
Therefore, under the assumption that $q\in q_0+ W^{2,\infty}(\Omega )$ and $\| q-q_0\|_{2,\infty}\leq m$, 
\[
\lambda _\ell |(f,\varphi _\ell)|\leq (1+\sqrt{\lambda _k})m +\lambda _k|(f,\varphi _\ell)|.
\]

This estimate in \eqref{h6} yields
\begin{align*}
\|f\|_2^2&\leq e^{2\lambda \tau}\|v(\cdot ,\tau)\|_2^2+\frac{2(1+\lambda_k)m^2+\lambda _k^2}{\lambda ^2}\sum_{\ell >N}|(f,\varphi _\ell )|^2
\\
&\leq e^{2\lambda \tau}\|v(\cdot ,\tau)\|_2^2+\frac{2(1+\lambda_k)m^2+\lambda _k^2}{\lambda ^2}\|f\|_2^2
\\
&\leq e^{2\lambda \tau}\|v(\cdot ,\tau)\|_2^2+\frac{C\lambda _k^2}{\lambda ^2}\|f\|_2^2
\\
&\leq e^{2\lambda \tau}\|v(\cdot ,\tau)\|_2^2+\frac{C\lambda _k^2}{\lambda ^2}.
\end{align*}
This inequality together with \eqref{h7} imply
\[
\|(q-q_0)\varphi _k\|_2^2=\|f\|_2^2\leq \sqrt{2}e^{2\tau ^2\lambda _k^2+2\lambda \tau}\|\partial _\nu u\|^2_{H^1((0,\tau ), L^2(\Gamma ))}+\frac{C\lambda _k^2}{\lambda ^2}.
\]
But
\[
\|\partial _\nu u\|_{H^1((0,\tau ), L^2(\Gamma ))}=\|\Lambda_q(\varphi )-\Lambda _{q_0}(\varphi _k)\|_{H^1((0,\tau ), L^2(\Gamma ))}\leq \|\Lambda_q-\Lambda _{q_0}\| \| \varphi _k\|_{H_0^1(\Omega )}\leq \sqrt{\lambda_k}\|\Lambda_q-\Lambda _{q_0}\|.
\]
Whence,
\begin{equation}\label{h8}
\|(q-q_0)\varphi _k\|_2^2=\|f\|_2^2\leq \sqrt{2}\lambda _ke^{2\tau ^2\lambda _k^2+2\lambda \tau} \|\Lambda_q-\Lambda _{q_0}\|^2+\frac{C\lambda _k^2}{\lambda ^2}.
\end{equation}

Now the usual way consists in minimizing, with respect to $\lambda$, the right hand side of the  inequality above. By a straightforward computation, one can see that the minimization argument is possible only if
\[
\frac{\lambda _ke^{-2\tau ^2\lambda _k^2}}{\|\Lambda_q-\Lambda _{q_0}\|^2}\gg 1.
\]
But this estimate does not guarantee that $\|\Lambda_q-\Lambda _{q_0}\|$ can be chosen arbitrarily small uniformly in $k$. However, the minimization argument works if we perturb $q_0$ by a finite dimensional subspace. That what we will discuss now.

\smallskip
Let $I>0$ be a given integer and $E_I=\textrm{span}\{\varphi_1,\ldots ,\varphi _I\}$. Since $|(q-q_0,\varphi _k)|^2\leq |\Omega |\|(q-q_0)\varphi _k\|_2^2$ by Cauchy-Schwarz's inequality, we get from \eqref{h8}
\[
\|q-q_0\|_2^2=\sum_{k=1}^I|((q-q_0),\varphi _k)|^2\leq C_I\left(e^{2\lambda \tau}\|\Lambda_q-\Lambda _{q_0}\|+\frac{1}{\lambda ^2}\right),
\]
for some constant $C_I$ depending on $I$. We observe that, according to the preceding analysis, $C_I$ surely blows-up when $I\rightarrow +\infty$.

\smallskip
Minimizing with respect to $\lambda$ the right hand side of the inequality above, we get
\[
\|q-q_0\|_2\leq C_I\left|\ln \left(\|\Lambda_q-\Lambda _{q_0}\|\right)\right|^{-1},
\]
provided that $\|\Lambda_q-\Lambda _{q_0}\|$ is sufficiently small. By a simple continuity argument, we see that $\|\Lambda_q-\Lambda _{q_0}\|$ is small whenever $\|q-q_0\|_1,\infty$ is small. If $\Lambda ^I_q=\Lambda _q{_{|E_I}}$, we end up getting

\begin{theorem}\label{theorem6.1}
Under the preceding notations and assumptions, there exist two constants $C_I$ and $\delta _I$ so that
\[
\|q-q_0\|_2\leq C_I\left|\ln \left(\|\Lambda ^I _q-\Lambda^I_{q_0}\|\right)\right|^{-1},
\]
if $\|q-q_0\|_{1,\infty}\le \delta _I$.
\end{theorem}

\begin{remark}\label{remark6.1}
We consider on $L^2(\Omega )$ the following norm, which weaker than its natural norm,
\[
\|w\|_\ast =\left(\sum_{k\geq 1} e^{-3\tau ^2\lambda _k^2}|(w,\varphi _k)|^2\right)^2,\;\; w\in L^2(\Omega ).
\]
Then \eqref{h8} yields
\[
\|q-q_0\|_\ast^2\leq \sqrt{2}e^{2\tau \lambda}\|\Lambda_q-\Lambda _{q_0}\|^2+\frac{C}{\lambda ^2}.
\]
We get by minimizing the right hand side with respect to $\lambda$
\[
 \|q-q_0\|_\ast \leq C\left|\ln \left(\|\Lambda _q-\Lambda_{q_0}\|\right)\right|^{-1}.
 \]
\end{remark}

\appendix
\section{}

\begin{proof}[Proof of Lemma \ref{lemmah1}]
We start be considering the following homogenous IBVP for the heat equation
\begin{equation}\label{he1}
\left\{
\begin{array}{lll}
 \partial _t u - \Delta u + q(x)u  = \phi \;\; &\mbox{in}\;   Q, 
 \\
u = 0 &\mbox{on}\;  \Sigma, 
\\
u(\cdot ,0) = \psi .
\end{array}
\right.
\end{equation}
Let $\phi\in L^2(Q)$ and $\psi \in H_0^1(\Omega )$. We obtain by applying one more time \cite[Theorem 1.43, page 27]{choulli} that the IBVP \eqref{he1} has a unique solution $u_{\phi ,\psi}\in H^{2,1}(Q)$ provided that $q\in L^\infty (\Omega )$. Moreover, for any $m>0$, there exists $C>0$ so that 
\begin{equation}\label{h9}
\|u_{\phi,\psi}\|_{H^{2,1}(Q)}\leq C\left(\|\psi\|_{H^1(\Omega )} +\|\phi\|_{L^2(Q)}\right),
\end{equation}
uniformly in $q\in mB_\infty$.

\smallskip
Let $A_q$ be the unbounded operator on $L^2(\Omega )$ defined by 
\[
A_q=-\Delta +q,\quad D(A_q)=H^2(\Omega )\cap H_0^1(\Omega ).
\]
Then the solution of \eqref{he1} is given by
\[
u_{\phi ,\psi}(\cdot ,t)=e^{-tA_q}\psi +\int_0^te^{-sA_q}\phi (\cdot ,t-s)ds,
\]
where $e^{-tA_q}$ is the semigroup generated by $-A_q$.

\smallskip
In the sequel, we write $u_\phi$ for $u_{\phi ,0}$.

\smallskip
Let $\phi \in C^1([0,\tau ]; L^2(\Omega ))$ so that $\phi (\cdot ,0)\in H_0^1(\Omega )$. Then
\[
\partial _tu_\phi (\cdot ,t)=e^{-tA_q}\phi(\cdot ,0)+\int_0^te^{-sA_q}\partial _t\phi (\cdot ,t-s)ds.
\]
In other words, $\partial _tu_\phi=u_{\partial _t\phi ,\phi(\cdot ,0)}$. Thus, estimate \eqref{h9} entails
\begin{equation}\label{h10}
\|\partial _tu_\phi\|_{H^{2,1}(Q)}\leq C\left(\|\phi (\cdot ,0)\|_{H^1(\Omega )} +\|\partial _t\phi\|_{L^2(Q)}\right).
\end{equation}

Next, let $\phi \in H^1((0,\tau );L^2(\Omega ))$ with $\phi (\cdot ,0)\in H_0^1(\Omega )$. Observing that $u_\phi =u_{\widetilde{\phi}}+u_{\phi(\cdot ,0 )}$, where $\widetilde{\phi}=\phi -\phi (\cdot ,0)$, and $\partial_tu_{\phi (\cdot ,0)}=u_{0,\phi (\cdot ,0)}$, we see that is sufficient to consider the case $\phi (\cdot ,0)=0$.

\smallskip
By density, there exist a sequence $(\phi_k)$ in $C_0^\infty ((0,\tau ];L^2(\Omega ))$ converging to $\phi \in H^1((0,\tau );L^2(\Omega ))$. Armed with \eqref{h9}, we get in a straightforward manner that $u_{\phi _k}$ and $u_{\partial _t\phi}$ converge respectively to $u_\phi$ and $u_{\partial _t\phi}$ in $H^{2,1}(Q)$. But, in light of the smoothness of $\phi_k$, $\partial _tu_{\phi _k}=u_{\partial _t\phi_k}$. Therefore, we have $\partial _tu_\phi=u_{\partial _t\phi}$ and \eqref{h10} holds true for all $\phi \in H^1((0,\tau );L^2(\Omega ))$ with $\phi (\cdot ,0)\in H_0^1(\Omega )$.

\smallskip
Now, let $q_0,q\in mB_\infty$ so that $q\in q_0+W^{1,\infty }(\Omega )$. Let $u_0\in H_0^1(\Omega )$. By an elementary computation, we get that $u:=u_q(u_0)-u_{q_0}(u_0)=u_\phi$, with $\phi =(q_0-q)u_{q_0}(u_0)$. Consequently, from the preceding discussion, $u,\partial _tu\in H^{2,1}(Q)$ and
\begin{equation}\label{h11}
\|u\|_{H^{2,1}(Q)} +\|\partial _tu\|_{H^{2,1}(Q)}\leq C\|u_0\|_{H^1(\Omega )}.
\end{equation}
For some constant $C=C(m)$.

\smallskip
Finally, using the continuity of the trace operator $w\in H^{2,1}(Q)\rightarrow \partial _\nu w\in L^2(\Upsilon )$, we obtain from \eqref{h11}
\[
\|\Lambda_q(u_0)-\Lambda_{q_0}(u_0)\|_{H^1((0,\tau );L^2(\Gamma ))}\leq C\|u\|_{H^1(\Omega )}.
\]
That is, we proved
\[
\|\Lambda_q-\Lambda_{q_0}\|\leq C,
\]
where $\|\Lambda_q-\Lambda_{q_0}\|$ is the norm of $\Lambda_q-\Lambda_{q_0}$ in $\mathscr{B}(H_0^1(\Omega ),H^1((0,\tau );L^2(\Gamma ))$.
\end{proof}

\bigskip

\end{document}